\documentclass[final,1p,times]{elsarticle}
\usepackage{amssymb,amsmath}

\journal{arXiv}

\newtheorem{teo}{Theorem}
\newtheorem{prop}[teo]{Proposition}
\newtheorem{cor}[teo]{Corollary}
\newtheorem{lema}[teo]{Lemma}
\newdefinition{defs}[teo]{Definition}
\newdefinition{ejemplo}[teo]{Example}
\newdefinition{ejemplos}[teo]{Examples}
\newdefinition{obs}[teo]{Remark}
\newdefinition{parr}[teo]{}
\newdefinition{notc}[teo]{Notation}
\newproof{proof}{Proof}

\def\Z{\mathbb{Z}}
\def\C{\mathbb{C}}
\def\N{\mathbb{N}}

\def\PP{\mathbb{P}}

\def\G{\mathcal{G}}

\def\then{\Longrightarrow}
\def\iff{\Longleftrightarrow}
\def\mono{\hookrightarrow}
\def\epi{\twoheadrightarrow}
\def\fl#1{\stackrel{#1}{\longrightarrow}}
\def\endo{\text{end}}
\def\auto{\text{aut}}
\def\der{\text{der}}
\def\rk{\text{rk}}
\def\im{\text{im}}

\begin{document}

\begin{frontmatter}

\title{Examples of varieties of structures.}
\author[dm]{C\'esar Massri\corref{correspondencia}\fnref{finanaciado}}
\address[dm]{Departamento de Matem\'atica, FCEN, Universidad de Buenos Aires, Argentina}
\cortext[correspondencia]{Address for correspondence: Departamento de Matem\'atica, FCEN, Universidad de Buenos Aires, Argentina}
\fntext[finanaciado]{The author was fully supported by CONICET, Argentina}
\ead{cmassri@dm.uba.ar}

\begin{abstract}
In this article we will introduce, among others,
the variety of subcomplexes and the variety of maps between complexes of given rank.
Also, varieties of $\mathfrak{g}$-structure like $\mathfrak{g}$-Grassmannian, $\mathfrak{g}$-determinantal varieties
and finally the variety $\mathcal{DGLA}(E)$ of differential graded Lie algebra structures on $E$.
We will compute the dimensions of these varieties and also some relevant properties.
The motivation of this article is to give examples of moduli spaces relevant to deformation theory.
\end{abstract}


\end{frontmatter}

\section*{Introduction.}

This article is related to the variety of complexes studied in \cite{2011arXiv1111.5514C}.
The variety of complexes, often called Buchsbaum-Eisenbud variety, is the main topic of several articles,
\cite{MR0314819}, \cite{MR0384817}, \cite{MR625334}, \cite{MR607117}, \cite{MR698481}, \cite{MR709861},
\cite{MR717615}, \cite{MR738917}, \cite{MR878472}, \cite{MR1020344}, \cite{MR1028732}, \cite{MR1119908},
\cite{MR1129131}, \cite{MR1783997}, \cite{MR1805470}, \cite{MR2774650}.\\

In \cite[\S 24]{MR0195995} there is a general construction. Let $V$ be a finite dimensional vector space over $\C$,
let $P_1,\ldots,P_m$ be homogeneous polynomial functions on $V$ and let $\mathcal{I}$ denote the homogeneous ideal
generated by $P_1,\ldots,P_m$ in the algebra $S^\star(V^\vee)$ of
all polynomial functions on $V$. Let $X\subseteq\PP V$ denote the projective variety associated to $\mathcal{I}$.
Let $G\subseteq GL(V)$ be the algebraic subgroup that leaves $\mathcal{I}$ stable. It follows that $G$ acts on $X$.
Denote by $\mathfrak{g}$ the Lie algebra of $G$. For each $x\in X$ we will define a structure of a complex $(d^0,d^1)$ in the
graded vector space $\mathfrak{g}\oplus V\oplus \C^{m}$. Let $d^0:\mathfrak{g}\rightarrow V$ be given by $d^0(a)=a.x$
and let $d^1:V\rightarrow\C^m$ be given by $d^1(v)=((dP_1)_x(v),\ldots,(dP_m)_x(v))$. It is an exercise to verify that
if $x\in X$ then $d^1d^0=0$, so we have defined a map from $X$ to the variety of complexes
on $\mathfrak{g}\oplus V\oplus \C^{m}$,
$$X\longrightarrow\{(d^0,d^1)\in\PP\hom(\mathfrak{g},V)\times\PP\hom(V,\C^m)\,|\,d^1d^0=0\}.$$
The homology of each complex has a geometric meaning. If $H^1(x)=0$, that is, if $\mathfrak{g}.x=T_xX$,
the point $x$ has an open orbit (it determines an irreducible component of $X$).
See \cite{MR0195995} for a proof and more references.\\

Let's see another interesting construction from
the theory of Differential Graded Lie Algebra (DGLA) in the context of
deformation theory. A Differential Graded Lie Algebra $(E,d,[-,-])$ is a graded vector space $E$ together with a differential $d$
of degree one and a graded bracket subject to some compatibilities,
\begin{itemize}
\item $[E^i,E^j]\subset E^{i+j}$, $[a,b] = -(-1)^{\overline{a}\overline{b}}[b,a]$
\item $[a,[b,c]]=[[a,b],c]+(-1)^{\overline{a}\overline{b}}[b,[a,c]]$
\item $d(E^i)\subset E^{i+1}$, $d^2=0$, $d[a,b]=[da,b]+(-1)^{\overline{a}}[a,db]$
\end{itemize}
Every DGLA $E$ comes with its Maurer-Cartan variety
$$M(E)=\{x\in E^1\,|\,dx+\frac{1}{2}[x,x]=0\}.$$
It is of interest to understand the local behavior of some $x\in M(E)$.
Note that if $x\in M(E)$ then $d+[x,-]$ is a new differential in the Graded Lie Algebra $(E,[-,-])$,
$$(d+[x,-])^2=d^2+[x,[x,-]]+d[x,-]+[x,d-]=$$
$$d^2+[x,[x,-]]+[dx,-]-[x,d-]+[x,d-]=\frac{1}{2}[[x,x],-]+[dx,-]=$$
$$[dx+\frac{1}{2}[x,x],-]=0.$$
Let $\mathcal{DGLA}(E)$ be the variety of DGLA structures on $E$ (we are assuming that $E$ is bounded)
and let $\mathcal{D}(E)$ be the
variety of complexes on the graded vector space $E$ then we have the following diagram,
$$M(E)\stackrel{\kappa}{\rightarrow}\mathcal{DGLA}(E)\stackrel{\pi}{\rightarrow}\mathcal{D}(E),$$
$$\kappa(x)=(d+[x,-],[-,-]),\quad \pi(\delta,[-,-])=\delta.$$
Again, we have a map to a variety of complexes but also, we have a map
to the mysterious variety $\mathcal{DGLA}(E)$.
From the definition of DGLA we know that $E^0$
is a Lie algebra and $E^i$ are $E^0$-modules, so to analyze $\mathcal{DGLA}(E)$ we would
need some auxiliaries varieties of structures.\\

Let's see another example. The theorem of Frobenius (\cite[2.32]{MR722297}) implies that every
differential ideal $\mathcal{I}\subseteq \Omega^\star_{\PP ^n}$ generated by $1$-forms
corresponds to a foliation in $\PP ^n$. The homogeneous polynomials of degree $e$ in $\mathcal{I}$
determines a subcomplex of the finite dimensional complex $\Omega^\star_{\PP ^n}(e)$ of forms with coefficients of degree $e$,
$$\mathcal{I}(e)\subseteq \Omega^\star_{\PP ^n}(e).$$
It would be necessary to understand the variety of subcomplexes of a given complex. In this article
we will do that.\\

This article is divided in two parts. In the first part (sections 1,2,3,4) we will work with
varieties of graded structures. In the second part (sections 5,6,7,8) we will work with
varieties of $\mathfrak{g}$-structures with $\mathfrak{g}$ a semisimple Lie algebra.
In the {\bf first section} we will give some preliminaries that we are going
to use in sections 2,3 and 4, like definitions, notations and some computations.
In the {\bf second, third and fourth sections} we will study the variety of subcomplexes of a complex,
the variety of maps of given rank and the variety of complexes of given dimensions.
The {\bf fourth section} is a review of \cite{2011arXiv1111.5514C}.
In the {\bf fifth section} we will give some definitions and notations to be used in
sections 6,7 and 8. In the {\bf sixth, seventh and eighth sections} we will study the
variety of maps of a given $\mathfrak{g}$-rank (we call it $\mathfrak{g}$-determinantal variety),
the variety of submodules of a given module (we call it $\mathfrak{g}$-Grassmannian) and
finally the variety $\mathcal{DGLA}(E)$ assuming $E=E^0\oplus E^1\oplus E^2$ with $E^0$ a
semisimple Lie algebra. All the proofs presented in this article uses standard techniques of
algebraic geometry, see \cite{MR1182558} for reference.

\section{Preliminaries on complexes.}
Given $(V,d)$ and $(W,\delta)$ two complexes, we will denote $\hom(V,W)$ the morphisms of complexes. It is a graduate vector space.
A morphism of degree $i$ is an element of $\hom^i(V,W)$. The group $GL(V)$ will
denote the automorphism of the complex $V$. An element in $GL(V)$ is an invertible endomorphism of $(V,d)$.
In this work we will assume that all the complexes are bounded and finite dimensional
$$W=\bigoplus_{i=0}^n W_i,\quad \dim W_i<\infty.$$

\begin{notc}
For $f\in\hom^0(V,L)$, let
$$\rk(f):=(\rk(f^0),\rk(f^1),\ldots,\rk(f^n))\in\N_0^{n+1}.$$

Also, let $\dim V:=\rk(1_V)$,
$$\dim V=(\dim(V^0),\dim(V^1),\ldots,\dim(V^n))\in\N_0^{n+1}.$$

For $r,s\in\N_0^{n+1}$ we define the order $s\leq r$ induced by $\N_0$,
$$s\leq r\iff s_i\leq r_i\quad 0\leq i\leq n.$$
\end{notc}

\begin{lema}\label{dim-hom^0}
Let $(V,d)$ and $(V',d')$ two complexes then $\hom^0(V,V')$
is a vector space of dimension
$$\dim\hom^0(V,V')=\sum_{i=0}^n h^ih'^{i}+\rk(d^i)\rk(d'^{i}),\quad h^i=\dim H^i(V),\,h'^{i}=\dim H^{i}(V').$$
\end{lema}
\begin{proof}
Let $Z^i:=\ker(d^i)$, $B^i:=\im(d^{i-1})$, $H^i:=Z^i/B^i$ and $C^i$ a complement of $Z^i$ in $V^i$, then
$$V^i=Z^i\oplus C^i=B^i\oplus H^i\oplus C^i.$$
Note that $C^i\cong B^{i+1}$ because $d|_{C}$ is injective.
Same for $V'$.\\

Let $f\in\hom^0(V,V')$ then $f(B)\subseteq B'$ and $f(H)\subseteq H'$, so we may suppose
$$B^i\oplus H^i\oplus B^{i+1}\fl{f^i}B'^{i}\oplus H'^{i}\oplus B'^{i+1}.$$
Given that $f^id^{i-1}=d'^{i-1}f^{i-1}$, we have that
$f^i|_{B^i}$ is determined by $f^{i-1}$, and then $f^i$ is defined by
an element of $\hom(H^i,H'^{i})\times \hom(B^{i+1},B'^{i+1})$. Finally,
$$\hom^0(V,V')\cong\bigoplus_{i=0}^n \hom(H^i,H'^{i})\times \hom(B^{i+1},B'^{i+1}).$$\qed
\end{proof}

\begin{cor}\label{dim-ker}
Let $L$ and $N$ be complexes such that $\dim L=\dim N=r$ and $\dim H(L)=\dim H(N)=s$, then
$$\dim\ker(d^i_L)=\dim\ker(d_N^i).$$
Even more, let $z^i:=\dim\ker(d^i)$ then
$$z^i=r_{i-1}-z^{i-1}+s_i.$$
\end{cor}
\begin{proof}
Let $Z^i=\ker(d^i)$, $B^i=\im(d^{i-1})$, $H^i=Z^i/B^i$ and
$C^i\cong B^{i+1}$ a complement of $Z^i$ in $V^i$.
The dimensions of these spaces are denoted with lower cases $z^i,b^i,h^i,c^i,v^i$.\\

Let proceed inductively. Given that the complexes start in degree zero, we have
$$z^0(L)=h^0(L)=h^0(N)=z^0(N).$$

Given that $\dim L^i=\dim N^i$ and $\dim L^i=z^i(L)+b^{i+1}(L)$
we have by inductive hypothesis $b^{i+1}(L)=b^{i+1}(N)$.
Given that $h^{i+1}(L)=h^{i+1}(N)$, we get
$$z^{i+1}(L)=z^{i+1}(N).$$\qed
\end{proof}

\begin{notc}
Let $\chi_i:\Z^{n+1}\fl{}\Z$ the $i$-partial characteristic,
$$\chi_i(t):=t_i-t_{i-1}+t_{i-2}-\ldots+(-1)^it_0,\quad t\in\Z^{n+1}$$
Resolving the recursion in \ref{dim-ker}, $z^i=r_{i-1}-z^{i-1}+s_i$, we have
$$z^i=\chi_{i-1}(r)+\chi_i(s),\quad r,s\in\N_0^{n+1}.$$
\end{notc}

\begin{prop}
Let $V$, $L$ and $N$ be complexes such that $\dim L=\dim N$ and $\dim H(L)=\dim H(N)$ then
$$\dim \hom^0(V,L)=\dim\hom^0(V,N).$$
\end{prop}
\begin{proof}
From \ref{dim-ker} we know $z^i(L)=z^i(N)$.
Given that $\dim(L^i)=\dim(N^i)$ we have $b^{i+1}(L)=b^{i+1}(N)$ then
$$h^i(V)h^i(L)+b^{i+1}(V)b^{i+1}(L)=h^i(V)h^i(N)+b^{i+1}(V)b^{i+1}(N).$$
Using \ref{dim-hom^0} we get $\dim\hom^0(V,L)=\dim\hom^0(V,N)$.\qed
\end{proof}

\begin{cor}\label{euler}
Let $(V,d)$ be a complex then $\chi_n(\dim V)=\chi_n(\dim HV)$.
\end{cor}
\begin{proof}
$$\chi_n(\dim V)=\sum_{i=0}^n(-1)^i\dim V^i=\sum_{i=0}^n(-1)^i(\dim \im(d^{i-1})+\dim H^i+\dim V/\im(d^{i-1}))=$$
$$\sum_{i=0}^n(-1)^i(\dim \im(d^{i-1})+\dim H^i+\dim \im(d^i))=\sum_{i=0}^n(-1)^i\dim H^i(V)=\chi_n(\dim HV).$$\qed
\end{proof}

\bigskip

The group $GL(V)\times GL(W)$ acts on the left on $\hom(V,W)$ in the following way
$$(\phi,\psi).f:=\psi f \phi^{-1},\quad (\phi,\psi)\in GL(V)\times GL(W).$$
Given that $f$ is a map of complexes, $\delta f=fd$, its conjugate is also a map of complexes
$$\delta \psi f\phi^{-1}=\psi\delta  f\phi^{-1}=\psi fd \phi^{-1}=\psi f \phi^{-1}d.$$
Given that $V$ and $W$ are bounded and finite dimensional in each degree,
$GL(V)\times GL(W)$ is a Lie group with Lie algebra $\hom^0(V,V)\times \hom^0(W,W)$.\\

The group $GL(V)$ acts on $\hom(V,V)$ by conjugation, $\phi.f:=\phi f\phi^{-1}$.

\section{Variety of subcomplexes of a given complex.}

Let us fix a complex $(W,\delta)$, $W=\bigoplus_0^n W_i$, $\dim W_i<\infty$ and
let $\mathcal{G}=\mathcal{G}(W)$ the variety of all the subcomplexes of $W$, that is,
$L\in\mathcal{G}$ if and only if $L\subseteq W$ is a subcomplex of $W$.\\

For $r,s\in\N_0^{n+1}$ let
$\G_{r,s}\subseteq\text{Grass}(W^0,r_0)\times\ldots\times \text{Grass}(W^{n},r_n)$
be the subvariety of $\mathcal{G}$ consisting of all the subcomplexes $L$ such that $\dim L=r$ and $\dim H(L)=s$,
$$\G_{r,s}(W)=\G_{r,s}:=\{(L^0,L^{1},\ldots,L^{n})\,|\,\delta^{i}(L^{i})\subseteq L^{i+1},\,\dim L=r,\,\dim H(L)=s\}.$$

\begin{prop}
Let $w:=\dim W$ and $h:=\dim H(W)$.
Suppose $\G_{r,s}\neq\emptyset$ then it is irreducible, smooth and
$$\dim \G_{r,s}=\sum_{i=0}^n(h_i-s_i)s_i+(\chi_i(w-h)-\chi_i(r-s))\chi_i(r-s).$$
\end{prop}
\begin{proof}
Consider the map that forgets the last coordinate $\G_{r,s}\fl{\pi}\widehat{\G}_{\widehat{r},\widehat{s}}$.
The fibers of this map $\pi^{-1}(L^0,\ldots,L^{n-1})$ are
$$\{L^n\subseteq W^{n}\,|\,\delta^{n-1}(L^{n-1})\subseteq
L^n,\,\dim(L^n)=r_n,\,\dim\ker(\delta^{n}|_{L^{n}})-\rk(\delta^{n-1}|_{L^{n-1}})=s_n\}\cong$$
$$\cong\text{Grass}(\ker\delta^n/\delta^{n-1}(L^{n-1}),s_n)\times\text{Grass}(W^{n}/\ker\delta^n,r_n-\rk(\delta^{n-1}|_{L^{n-1}})-s_n).$$
They are smooth and irreducible. Given that
$$r_n-\rk(\delta^{n-1}|_{L^{n-1}})-s_n=r_n-(\dim\ker(\delta^{n}|_{L^{n}})-s_n)-s_n=$$
$$r_n-\chi_{n-1}(r)-\chi_n(s)=\chi_n(r)-\chi_n(s)=\chi_n(r-s),$$
we have
$$\dim\pi^{-1}(L^0,\ldots,L^{n-1})=$$
$$(\dim H^n(W)-s_n)s_n+(\dim W^n-\dim\ker\delta^n-\chi_n(r-s))\chi_n(r-s)=$$
$$(h_n-s_n)s_n+(\chi_n(w-h)-\chi_n(r-s))\chi_n(r-s).$$
The result follows by induction.\qed
\end{proof}

\bigskip

The group $GL(W)$ acts on $\G(W)$.
We will say that $L\sim N$ if there exist $\phi\in GL(W)$ such that $\phi|_L:L\fl{}N$,
in other words, $L$ is related to $\phi(L)$.
The relation in well defined because $\phi$ is an invertible map of complexes and given that $\delta(L)\subseteq L$ we have
$$\delta\phi(L)=\phi\delta(L)\subseteq\phi(L)\then\phi.L\in\G(W).$$
This action preserve the variety $\G_{r,s}$,
$$\dim(L)=\dim(\phi.L),\quad \dim(H(L))=\dim(H(\phi.L)).$$
Given that the action is transitive on $\G_{r,s}$, it is an orbit.

\bigskip

Let's study the existence of some $L\subseteq W$ with $\dim L=r$ and $\dim H(L)=s$,
in other words, $\G_{r,s}(W)\neq\emptyset$.
\begin{prop}
Let $r\leq w:=\dim W$, $s\leq h:=\dim H(W)$, $h\leq w$ and $s\leq r$.
$$\G_{r,s}(W)\neq\emptyset\iff 0\leq\chi_i(r-s)\leq\chi_i(w-h).$$
\end{prop}
\begin{proof}
Write each $W^i=Z^i\oplus C^i$, where $Z^i\cong B^i\oplus H^i$ and $C^i\cong B^{i+1}$.
Let's start the induction.
Define $L^0=L_Z^0\oplus L_C^0$, where $L_Z^0\subseteq Z^0=H^0$ and $L_C^0\subseteq C^0$ with
$\dim L_Z^0=s_0$ and $\dim L_C^0=r_0-s_0$. This is possible because
$$0\leq r_0-s_0=\dim(L_C^0)\leq\dim C^0=w_0-h_0\then\dim(L^0)=r_0.$$

Suppose we have defined up to $L^i$. Given that $\delta^{i}|_{C^i}$ is injective
there will be in $L^{i+1}$ a copy of $L^i\cap C^i\cong L^{i+1}\cap B^{i+1}$.
By inductive hypothesis, $L$ is a complex of the required dimensions, then
$$\dim(L^{i+1}\cap B^{i+1})=\dim(L^i\cap C^i)=\dim(L^i)-\dim(Z^i\cap L^i)=r_i-\chi_{i-1}(r)-\chi_i(s)=\chi_i(r-s).$$
Let $L_H^{i+1}\subseteq H^{i+1}$ of dimension $s_{i+1}$, let $L_B^{i+1}=L^{i+1}\cap B^{i+1}$
and let $L_C^{i+1}\subseteq C^{i+1}$ of dimension
$$r_{i+1}-(\dim L_H^{i+1} + \dim L_B^{i+1})=r_{i+1}-s_{i+1}-\chi_i(r-s)=\chi_{i+1}(r-s).$$
This is possible because
$$0\leq\chi_{i+1}(r-s)=\dim L_C^{i+1}\leq \dim C^{i+1}=\chi_{i+1}(w-h)\then$$
$$L^{i+1}:=L_H^{i+1}\oplus L_B^{i+1}\oplus L_C^{i+1},\quad
\dim L^{i+1}=r_{i+1},\,\dim H^{i+1}(L)=s_{i+1}.$$
The other implication is obvious.\qed
\end{proof}

\section{Variety of maps between complexes of given rank.}

Let $V$ and $W$ be complexes,
$$V=\bigoplus_{i=0}^n V_i,\, W=\bigoplus_{i=0}^n W_i,\quad \dim V_i,\,\dim W_i<\infty.$$
The group $GL(V)\times GL(W)$ acts on $\hom^0(V,W)$,
$$\psi f\phi^{-1}=(\psi^0 f^0\phi^{0,-1},\psi^1 f^1\phi^{1,-1},\ldots,\psi^n f^n\phi^{n,-1})$$
where $f^{i+1}d^i=\delta^{i} f^i$, $\phi^{i+1}d^i=d^i \phi^i$, $\psi^{i+1}\delta^i=\delta^i \psi^i$,
$\phi^i\in GL(V^i)$ and $\psi^i\in GL(W^i)$.
Note that $\rk(f)$ is invariant under conjugation, $\rk(\psi f\phi^{-1})=\rk(f)$ and
if we denote $[f]$ the map induced on homology, $\rk([f])$ is also invariant
by this action, $\rk([\psi f\phi^{-1}])=\rk([f])$.
For each $r,s\in\N^{n+1}$ let
$$\mathcal{C}_{r,s}=\mathcal{C}_{r,s}(V,W):=\{f\in\hom^0(V,W)\,|\,\rk(f)=r,\,\rk([f])=s\}.$$
Then
$$\bigcup_{\substack{0\leq r\leq\dim W,\\0\leq s\leq\dim H(W)}}\mathcal{C}_{r,s}=\hom^0(V,W).$$

\begin{lema}
Let $f,g\in\hom^0(V,W)$ such that $\rk(f)=\rk(g)=r$ and $\rk([f])=\rk([g])=s$ then
$$\dim f^i(\ker d^i)=\chi_{i-1}(r)+\chi_{i}(s)=\dim g^i(\ker\delta^i).$$
\end{lema}
\begin{proof}
Given that $\rk(f^0|_{\ker d^0})=\rk([f^0])=\rk([g^0])=\rk(g^0|_{\ker\delta^0})$ we have
$$\dim(f^0(\ker d^0))=\dim(g^0(\ker\delta^0)).$$
The following applies
$$\rk(f^i)=\dim f^i(\ker d^i)+\rk(f^{i+1}d^i),\quad \rk(g^i)=\dim g^i(\ker \delta^i)+\rk(g^{i+1}\delta^i).$$
By inductive hypothesis, $\dim f^i(\ker d^i)=\dim g^i(\ker \delta^i)$, also $\rk(f^i)=\rk(g^i)$ then
$$\rk(f^{i+1}d^i)=\rk(g^{i+1} \delta^i).$$
Given that $\rk([f^{i+1}])=\rk([g^{i+1}])$ we finally get
$$\dim f^{i+1}(\ker d^{i+1})=\rk(f^{i+1}d^i)+\rk[f^{i+1}]=\rk(g^{i+1} \delta^i)+\rk[g^{i+1}]=\dim g^{i+1}(\ker\delta^{i+1}).$$
The formula for the dimension follows from the following recursion:
$$\rk f^i=\dim f^i(\ker d^i)+\rk(f^{i+1} d^{i})\then$$
$$\rk f^i+\rk [f^{i+1}]=(\dim f^i(\ker d^i)+\rk (f^{i+1} d^{i}))+\rk [f^{i+1}]\then$$
$$r_i+s_{i+1}=\dim f^i(\ker d^i)+(\rk (f^{i+1} d^{i})+\rk [f^{i+1}])\then$$
$$r_i+s_{i+1}=\dim f^i(\ker d^i)+\dim f^{i+1}(\ker d^{i+1}).$$\qed
\end{proof}

\begin{prop}
Let $v=\dim V$, $h^v=\dim H(V)$, $w=\dim W$ and $h^w=\dim H(W)$.
If $\mathcal{C}_{r,s}\neq\emptyset$ then $\mathcal{C}_{r,s}$ is irreducible, smooth and
$$\dim(\mathcal{C}_{r,s})=\sum_{i=0}^n(h^v_i+h^w_i-s_i)s_i+(\chi_i(v-h^v)+\chi_i(w-h^w)-\chi_i(r-s))\chi_i(r-s).$$
\end{prop}
\begin{proof}
Consider the map that forgets the last morphism
$$\pi:\mathcal{C}_{r,s}\epi\widehat{\mathcal{C}}_{\widehat{r},\widehat{s}}$$
Its fibers $\pi^{-1}(\widehat{f})=\pi^{-1}(f^0,f^1,\ldots,f^{n-1})$ are
$$\{f^n\in\hom(V^n,W^{n})\,|\,f^nd^{n-1}=\delta^{n-1}f^{n-1},\,\rk(f^n)=r_n,\rk([f^n])=s_n\}.$$
As usual write $V^n=B^n_V\oplus H^n_V\oplus C^n_V$ and $W^n=B^n_W\oplus H^n_W\oplus C^n_W$.
Note that
$$\{f^n\in\hom(V^n,W^{n})\,|\,f^nd^{n-1}=\delta^{n-1}f^{n-1}\}\cong\hom(H^n_V,H^n_W)\times\hom(C^n_V,C^n_W)$$
The rank of $f^n$ over $H^n_V$ must be $s_n$ and over $C^n_V$ must be
$$r_n-\rk(\delta^{n-1}f^{n-1})-s_n=\chi_n(r)-\chi_n(s)=\chi_n(r-s).$$
Then all the fibers $\pi^{-1}(\widehat{f})$ are isomorphic to
$$\{f_1\in\hom(H^n_V,H^n_W)\,|\,\rk(f_1)=s_n\}\times\{f_2\in\hom(C^n_V,C^n_W)\,|\,\rk(f_2)=\chi_n(r-s)\}.$$
They are irreducible, smooth and of dimension (\cite[proposición 12.2]{MR1182558})
$$(\dim H^n_V+\dim H^n_W-s_n)s_n+(\dim C^n_V+\dim C^n_W-\chi_n(r-s))\chi_n(r-s)=$$
$$(h^v_n+h^w_n-s_n)s_n+(\chi_n(v-h^v)+\chi_n(w-h^w)-\chi_n(r-s))\chi_n(r-s).$$
The result follows by induction.\qed
\end{proof}

\bigskip

Given that $GL(V)\times GL(W)$ acts transitively on
the irreducible, smooth varieties
$\mathcal{C}_{r,s}$, it follows that $\mathcal{C}_{r,s}$ are the orbits of the action.

\bigskip

\begin{prop}
Let $w:=\dim W$ and $h_w:=\dim H(W)$.
Assume $\mathcal{C}_{r,s}\neq\emptyset$ then
$$\overline{\mathcal{C}_{r,s}}=\bigcup_{(u,t)\leq(s,r)}\mathcal{C}_{t,u}=\{f\in\hom^0(V,W)\,|\,\rk(f)\leq r,\,rk([f])\leq s\}.$$
\end{prop}
\begin{proof}
Let
$$E_{r,s}:=\bigcup_{(u,t)\leq(s,r)}\mathcal{C}_{t,u}.$$
Being close we have $\overline{\mathcal{C}_{r,s}}\subseteq E_{r,s}$. Given that
$\mathcal{C}_{r,s}$ is dense in $E_{r,s}$, to get the equality we will see that
$E_{r,s}$ is irreducible. Let
$$\mathcal{I}_{r,s}:=\{(L,f)\,|\,\im(f)\subseteq L\}\subseteq \G_{r,s}(W)\times\hom^0(V,W).$$
The fibers of the first projection $\pi_1^{-1}(L)\cong\hom^0(V,L)$ are irreducible of the same dimension,
then $\mathcal{I}_{r,s}$ and $\im(\pi_2)$ are irreducible. Let's see that $\im(\pi_2)=E_{r,s}$,
let $f\in\mathcal{C}_{t,u}$, $(t,u)\leq (r,s)$ then
$$\exists L\in \G_{r,s}(W)\,|\,\im(f)\subseteq L\iff\G_{r-t,s-u}(W/\im(f))\neq\emptyset\iff$$
$$\chi_i( (r-t)-(s-u) )\leq\chi_i( (w-t)-(h_w-u) )\iff$$
$$\chi_i(r-s)\leq\chi_i(w-h_w)\iff\G_{r,s}(W)\neq\emptyset.$$
Given that $\mathcal{C}_{r,s}\neq\emptyset$, we get $\G_{r,s}(W)\neq\emptyset$.
Note that we recover the formula
$$\dim(\mathcal{C}_{r,s})=\dim(E_{r,s})=\dim(\G_{r,s}(W))+\dim(\hom^0(V,L)).$$\qed
\end{proof}

\begin{cor}
$\mathcal{C}_{t,u}\subseteq\overline{\mathcal{C}_{r,s}}\iff (u,t)\leq (s,r)$.\qed
\end{cor}

\begin{cor}
$$\overline{\mathcal{C}_{t,u}}\bigcap\overline{\mathcal{C}_{r,s}}=
\overline{\mathcal{C}_{\min(t,r),\min(u,s)}}$$
where $\min(a,b):=(\min(a_0,b_0),\ldots,\min(a_n,b_n))$.\qed
\end{cor}

\bigskip

\begin{prop}\label{comp_en_hom}
Let $v=\dim V,\,w=\dim W,\,h^v=\dim H(V),\,h^w=\dim H(W)$.
Let $r$ and $s$ such that $r\leq \min(v,w)$, $s\leq \min(h^v,h^w)$ and $s\leq r$.
$$\mathcal{C}_{r,s}\neq\emptyset\iff 0\leq\chi_i(r-s)\leq \min(\chi_i(v-h^v),\chi_i(w-h^w))$$
\end{prop}
\begin{proof}
As usual write $V^i=Z^i_V\oplus C^i_V$, $W^i=Z^i_W\oplus C^i_W$, $Z^i_V=B^i_V\oplus H^i_V$, $Z^i_W=B^i_W\oplus H^i_W$.
Let's proceed by induction. Let $f^0=f_Z^0\oplus f_C^0$ of rank $s_0$ and $r_0-s_0$ respectively,
$$H^0_V=Z^0_V\fl{f_Z^0}Z^0_W=H^0_W,\quad C^0_V\fl{f_C^0}C^0_W.$$
This is possible because
$$0\leq r_0-s_0=\rk(f^0_C)\leq\min(\dim C^0_V,\dim C^0_W)\then \rk(f^0)=r_0.$$
Suppose we have defined $f^i$.
Given that $d^i|_{C^i_V}$ and $\delta^i|_{C^i_W}$ are injective, there will be in $f^{i+1}$
the component $f^i_C$ of rank $\rk(f^i)-\rk(f^i|_{Z^i_V})$. Let's call it $f^{i+1}_B$.
By inductive hypothesis we have
$$\rk(f^{i+1}_B)=\rk(f^i)-\rk(f^i|_{Z^i})=r_i-\chi_{i-1}(r)-\chi_i(s)=\chi_i(r)-\chi_i(s).$$
Let $f_H^{i+1}$ of rank $s_{i+1}$ and let $f_C^{i+1}$ of rank
$$r_{i+1}-s_{i+1}-(\chi_i(r)-\chi_i(s))=\chi_{i+1}(r)-\chi_{i+1}(s).$$
This is possible because
$$0\leq\chi_{i+1}(r)-\chi_{i+1}(s)\leq \min(\dim C^{i+1}_V,\dim C^{i+1}_W)=\min(\chi_{i+1}(v-h^v),\chi_{i+1}(w-h^w)).$$
Then $\rk(f^{i+1})=r_{i+1}$ and $\rk([f]^{i+1})=s_{i+1}$. The other implication is obvious.\qed
\end{proof}

\begin{cor}
Let $f\in\hom^0(V,W)$ then
$$\chi_n(\rk(f))=\chi_n(\rk([f])).$$
\end{cor}
\begin{proof}
We have $f\in\mathcal{C}_{t,u}$ for some $t,u\in\N_0^{n+1}$.
From \ref{comp_en_hom} and \ref{euler} follows
$$0\leq\chi_n(t-u)\leq 0\then \chi_n(t)=\chi_n(u).$$\qed
\end{proof}

\subsection{Subvariety of quasi-isomorphism.}

Fix $(V,d)$ and $(W,\delta)$ two complexes such that $h=\dim H(V)=\dim H(W)$.
This is a necessary hypothesis for the existence of quasi-isomorphisms.
Let $\mathcal{Q}=\mathcal{Q}(V,W)\subseteq \hom^0(V,W)$ be the subvariety of quasi-isomorphisms,
$v:=\dim V$ and $w:=\dim W$. Let $f\in\hom^0(V,W)$ be
a quasi-isomorphism then there exist $q$ such that $f\in\mathcal{C}_{q,h}$.
In fact, using \ref{comp_en_hom}, we know that
$$\mathcal{C}_{q,h}\neq\emptyset\iff 0\leq \chi_i(q-h)\leq \min(\chi_i(v-h),\chi_i(w-h))\iff$$
$$0\leq \chi_i(q)-\chi_i(h)\leq \min(\chi_i(v),\chi_i(w))-\chi_i(h)\iff$$
$$\chi_i(h)\leq \chi_i(q)\leq \min(\chi_i(v),\chi_i(w)).$$
Then if $q$ is such that
$$h\leq q\leq\min(v,w),\quad \chi_i(h)\leq \chi_i(q)\leq \min(\chi_i(v),\chi_i(w))
\then\overline{\mathcal{C}_{q,h}}\subseteq\mathcal{Q}.$$
The irreducible components of $\mathcal{Q}$ correspond to the maximal $q$ satisfying this condition.
When $v=w$ we have one maximum $q$, so $\mathcal{Q}=\overline{\mathcal{C}_{v,h}}$ is irreducible of dimension
$$\dim(\mathcal{Q})=\sum_{i=0}^n h_i^2+\chi_i(v-h)^2.$$

\section{Variety of complexes.}
Suppose we have $V=H(V)$. We want to study the variety
$\mathcal{D}=\mathcal{D}(V)\subseteq\hom^1(V,V)$ of all the complex structures over $V$.
$$\mathcal{D}=\mathcal{D}(V):=\{d\in\hom^0(V,V[1])\,|\,d\circ d=0\}.$$
The action of $GL(V)$ on $\mathcal{D}$ is different from the action of the group
$GL(V)\times GL(V[1])$ on $\hom^1(V,V)$ whose orbits are $\mathcal{C}_{r,r}$.
The group $GL(V)$ preserve the condition $d^2=0$ but $GL(V)\times GL(V[1])$ does not.
Let's define
$$\mathcal{D}_r:=\{d\in\mathcal{C}_{r,r}\,|\,d^2=0\}=\{d:V\fl{}V[1]\,|\,d^2=0,\,\rk(d)=r\}.$$
The reason why we used $\mathcal{C}_{r,r}$ in the definition is
that $V=H(V)$, so $\rk(d)=\rk([d])$ for every $d\in\hom^0(V,V[1])$.
The action of $GL(V)$ preserve $\mathcal{D}_{r}$,
$$(\phi d\phi^{-1})^2=\phi d\phi^{-1}\phi d\phi^{-1}=\phi d^2\phi^{-1}=0.$$
Given that $d^{i+1}d^i=0$ it is possible to find a basis of $V^i$ and of $V^{i+1}$
in such a way that the matrixes of both maps are diagonal (with ones and zeros),
then the action is transitive implying that $\mathcal{D}_{r}$ are the orbits.
Even more, the irreducible components of $\mathcal{D}$ are $\overline{\mathcal{D}_{r}}$
for some $r\in\N_0^{n+1}$.

\begin{prop}
Let $v=\dim V$ and let $r\in\N^{n+1}_0$.
If $\mathcal{D}_r\neq\emptyset$ then $\mathcal{D}_r$ is smooth, irreducible and
$$\dim \mathcal{D}_r=\sum_{i=0}^n (v_i+v_{i+1} -r_i-r_{i-1})r_i.$$
\end{prop}
\begin{proof}
Consider the map that forget the last differential
$$\pi:\mathcal{D}_{r}\epi\widehat{\mathcal{D}}_{\widehat{r}}$$
The fibers $\pi^{-1}(\widehat{d})=\pi^{-1}(d^0,d^1,\ldots,d^{n-1})$ are
$$\{d^n\in\hom(V^n,V^{n+1})\,|\,d^nd^{n-1}=0,\,\rk(d^n)=r_n\}\cong$$
$$\{f\in\hom(V^n/\im(d^{n-1}),V^{n+1})\,|\,\rk(f)=r_n\}.$$
They are irreducible, smooth and of dimension
$$((v_n-r_{n-1})+v_{n+1}-r_n)r_n=(v_n+v_{n+1}-r_n-r_{n-1})r_n.$$
The result follows by induction.\qed
\end{proof}

\begin{prop}
Let $v=\dim V$ and let $r\in\N^{n+1}_0$.
Assume $\mathcal{D}_r\neq\emptyset$ then
$$\overline{\mathcal{D}_r}=\bigcup_{t\leq r}\mathcal{D}_t=\{d\,|\,d^2=0,\,\rk(d)\leq r\}.$$
\end{prop}
\begin{proof}
Let
$$E_r:=\{d:V\fl{}V[1]\,|\,d\circ d=0,\,\rk(d)\leq r\}.$$
Being close we have $\overline{\mathcal{D}_r}\subseteq E_r$. Let
$$\mathcal{I}_r:=\{(L,d)\,|\,\im\,d^i\subseteq L^{i+1}\subseteq\ker d^{i+1}\}\subseteq \G_{r,r}(V)\times\mathcal{D}.$$
The fibers of the first projection $\pi_1^{-1}(L)\cong\hom^0(V/L,L[1])$ are irreducible of the same dimension,
then $\mathcal{I}_r$ and $\im(\pi_2)=E_r$ are irreducible.
Given that $\mathcal{D}_r$ is dense in $E_r$ we get the equality.\qed
\end{proof}

\begin{cor}
$\mathcal{D}_t\subseteq\overline{\mathcal{D}_r}\iff t\leq r$.\qed
\end{cor}

\begin{cor}
$\overline{\mathcal{D}_t}\bigcap\overline{\mathcal{D}_r}=\overline{\mathcal{D}_{\min(t,r)}}.$\qed
\end{cor}

\begin{prop}\label{comp_en_D}
Let $v=\dim V$ and let $r\in\N^{n+1}_0$ be such that $r_i\leq v_{i+1}$ then
$$\mathcal{D}_{r}\neq\emptyset\iff r_{i}+r_{i-1}\leq v_{i}.$$
Even more if $d\in\mathcal{D}_{r}$,
$$\dim H^i(V,d)=v_i-r_i-r_{i-1}.$$
\end{prop}
\begin{proof}
Suppose we have $d^0,\ldots,d^{i-1}$.
Given that $d\circ d=0$, the map $d^i$ induces
$$d^i:V^i/\im(d^{i-1})\fl{}V^{i+1}.$$
We need $\rk(d^i)=r_i$.
By hypothesis we have $r_i\leq v_{i+1}$, so $d^i$ exist if and only if $r_i\leq v_i-r_{i-1}$.
In particular we have calculated the dimension of $H^i(V,d)$.\qed
\end{proof}

\subsection{Subvariety of exact complexes.}

Let $\mathcal{E}=\mathcal{E}(V)\subseteq \mathcal{D}(V)$ be the subvariety of exact complexes.
Let $v=\dim(V)$ and $e_i:=\chi_i(v)\geq 0$, from \ref{comp_en_D} we have
$$e_i+e_{i-1}=\chi_i(v)+\chi_{i-1}(v)=v_i\leq v_i\then\overline{\mathcal{D}_{e}}\subseteq\mathcal{D}.$$
$$\dim (\mathcal{D}_{e})=\sum_{i=0}^n (v_i+v_{i+1}-e_i-e_{i-1})e_i.$$
Given that $v_i-e_{i-1}-e_i=0$, if $d\in \mathcal{D}_{e}$ we have $\dim H(V,d)=0$.
It follows from \ref{comp_en_D} that any exact differential belongs to $\mathcal{D}_{e}$, then
$$\mathcal{E}=\overline{\mathcal{D}_{e}}=\{d\,|\,d^2=0,\,\rk(d^i)\leq e_i\}.$$
If some $e_i<0$ there are no exact differentials.

\section{Preliminaries on $\mathfrak{g}$-structures.}

Let $\mathfrak{g}$ be a semisimple Lie algebra over $\C$.
Let $R(\mathfrak{g})$ be the ring of finite dimensional
$\mathfrak{g}$-representations.
To each representation $V$ we associate its class $cl(V)\in R(\mathfrak{g})$.
Every finite dimensional representation could be written as a sum of irreducible representation.
This decomposition in unique up to isomorphisms (see \cite[\S 23.2]{MR1153249}).
For example, it is well known, \cite[11.31]{MR1153249}, that if $V=S^r(\C^2)$ then
$$cl(S^2(V))=\sum_{m=0}^{\lfloor\frac{r}{2}\rfloor} cl(S^{2r-4m}(\C^2) )\in R(\mathfrak{sl}_2(\C)).$$

\begin{defs}
For $S\in R(\mathfrak{g})$ we will say $S\geq 0$ if its decomposition as
a finite sum of irreducible representations has only non-negative coefficients, in particular
$$S_1\leq S_2\iff S_2-S_1\geq 0.$$
Note that $V\subseteq W\iff cl(V)\leq cl(W)$ where $cl(V)$ and $cl(W)$ are its classes in $R(\mathfrak{g})$.\\

Given $f:V\fl{}W$ a morphism let's define its $\mathfrak{g}$-rank as
$$\rk_{\mathfrak{g}}(f):=cl(\im(f))\in R(\mathfrak{g})\then\rk_{\mathfrak{g}}(f)\leq cl(W).$$

Let $\hom_{\mathfrak{g}}(V,W)_{s}$ be the variety of all the $\mathfrak{g}$-morphisms
with $\mathfrak{g}$-rank $s\in R(\mathfrak{g})$.
Given a representation $V$ and a subrepresentation $S$, let
$Gr_{\mathfrak{g}}(V,S)$ be the variety of all the subrepresentations isomorphic to $S$.
Given that this variety only depends on the classes of $V$ and of $S$, in general
we will denote it $Gr_{\mathfrak{g}}(cl(V),cl(S))$.
\end{defs}

We are going to study $\hom_{\mathfrak{g}}(V,W)_{s}$ and $Gr_{\mathfrak{g}}(V,S)$.
Let's start with $\hom_{\mathfrak{g}}(V,W)_{s}$.

\section{$\mathfrak{g}$-determinantal variety.}

\begin{prop}
Let $V$ and $W$ be two representations with
$cl(V)=\sum_{i=1}^k n_icl(V_i)$ and $cl(W)=\sum_{i=1}^k m_icl(V_i)$ then
$$\dim\hom_{\mathfrak{g}}(V,W)=\dim\hom_{\mathfrak{g}}(cl(V),cl(W))=\sum_{i=1}^k n_im_i.$$
\end{prop}
\begin{proof}
It follows from the following:
$$\hom_{\mathfrak{g}}(V,W)=\hom_{\mathfrak{g}}(\oplus V_i^{n_i},\oplus V_j^{m_j})=\oplus_{i,j}\hom_{\mathfrak{g}}(V_i^{n_i},V_j^{m_j})=$$
$$\oplus_{i,j}\hom_{\mathfrak{g}}(V_i,V_j)^{n_i m_j}=\oplus_{i}\hom_{\mathfrak{g}}(V_i,V_i)^{n_im_i}=\oplus_{i=1}^k\C^{n_im_i}.$$
The last equality follows from Schur's Lemma, \cite[\S 6.1]{MR499562}.\qed
\end{proof}

\begin{lema}
Let $V$ be an irreducible representation and let $n,m,s\in\N$ with $s\leq\min(m,n)$ then
$$\hom_{\mathfrak{g}}(V^n,V^m)_{V^s}\cong\hom(\C^n,\C^m)_s.$$
\end{lema}
\begin{proof}
The representation $V$ has only one highest weight line, $\langle v\rangle$. In general $V^n$
has only $n$ highest weight vectors linearly independent $\{v_1,\ldots,v_n\}$. A morphism from $V^n$ is determined
by this vectors. Every morphism sends highest weight vectors to highest weight vectors hence
$$\hom_{\mathfrak{g}}(V^n,V^m)=\hom(\C^n,\C^m).$$
Let $f\in\hom_{\mathfrak{g}}(V^n,V^m)$ such that $\im(f)\cong V^s$ then
$f$ sends the highest weight vectors $v_1,\ldots,v_n$ to $s$ linearly independent
highest weight vectors of $V^m$. In other words, $f$ determine a linear map of rank $s$.
Finally the result follows.\qed
\end{proof}

\begin{cor}
Let $V$ and $W$ be two representations with
$cl(V)=\sum_{i=1}^k n_icl(V_i)$, $cl(W)=\sum_{i=1}^k m_icl(V_i)$ and let $0\leq s\leq \min(cl(V),cl(W))$
with $s=\sum_{i=1}^k s_icl(V_i)$ then
$$\hom_\mathfrak{g}(V,W)_s=\hom(\C^{n_1},\C^{m_1})_{s_1}\times\ldots\times\hom(\C^{n_k},\C^{m_k})_{s_k}.$$
In particular, to have maximal $\mathfrak{g}$-rank is a generic condition.
\end{cor}
\begin{proof}
It follows from the previous lemma and from the following:
$$\hom_\mathfrak{g}(V,W)_s=
\hom_\mathfrak{g}(V_1^{n_1},V_1^{m_1})_{V_1^{s_1}}\times\ldots
\times\hom_\mathfrak{g}(V_k^{n_k},V_k^{m_k})_{V_k^{s_k}}.$$\qed
\end{proof}

\begin{prop}
Given $cl(V)=\sum_{i=1}^k n_icl(V_i)$, $cl(W)=\sum_{i=1}^k m_icl(V_i)$ and
$s=\sum_{i=1}^k s_icl(V_i)$ such that $0\leq s\leq \min(cl(V),cl(W))$ we have that the variety
$\hom_{\mathfrak{g}}(V,W)_{s}$ is irreducible, smooth and
$$\dim\hom_{\mathfrak{g}}(V,W)_{s}=\sum_{i=1}^k (n_i+m_i-s_i)s_i.$$
\end{prop}
\begin{proof}
We already know that the variety $\hom_{\mathfrak{g}}(V,W)_{s}$ is irreducible and smooth.
The dimension follows from \cite[12.2]{MR1182558},
$$\dim \hom(\C^{n_i},\C^{m_i})_{s_i}=(n_i+m_i-s_i)s_i.$$\qed
\end{proof}

\begin{obs}
Assume $V=W$ and $cl(V)=\sum_{i=1}^k n_i cl(V_i)$.
An endomorphism of maximal $\mathfrak{g}$-rank is the same as an automorphism,
so we have that the space of automorphisms is a dense open subset of $\endo_{\mathfrak{g}}(V)$.
In particular,
$$\dim\auto_{\mathfrak{g}}(V)=\sum_{i=1}^k n_i^2.$$
\end{obs}

\begin{prop}
The group $\auto_{\mathfrak{g}}(V)\times\auto_{\mathfrak{g}}(W)$ acts on the left on
$\hom_{\mathfrak{g}}(V,W)$ and this action is transitive on $\hom_{\mathfrak{g}}(V,W)_s$.
\end{prop}
\begin{proof}
Given $(\phi,\psi)\in\auto_{\mathfrak{g}}(V)\times\auto_{\mathfrak{g}}(W)$ and $f\in \hom_{\mathfrak{g}}(V,W)_s$, let
$$(\phi,\psi).f:=\phi f \psi^{-1}$$
The automorphisms $\phi$ and $\psi$ sends each irreducible submodule
into a copy of that irreducible submodule, then
$$\im(f)\cong \im(\phi f \psi^{-1})\then\rk_{\mathfrak{g}}(f)=\rk_{\mathfrak{g}}(\phi f \psi^{-1}).$$
The transitivity follows from standard arguments.\qed
\end{proof}

\begin{prop}
Given $cl(V)=\sum_{i=1}^k n_icl(V_i)$, $cl(W)=\sum_{i=1}^k m_icl(V_i)$ and
$s=\sum_{i=1}^k s_icl(V_i)$ such that $0\leq s\leq \min(cl(V),cl(W))$ we have
$$\overline{\hom_{\mathfrak{g}}(V,W)_s}=
\bigcup_{t\leq s}\hom_{\mathfrak{g}}(V,W)_t=\{f:V\fl{}W\,|\,\rk_{\mathfrak{g}}(f)\leq s\}.$$
\end{prop}
\begin{proof}
The result follows from
$$\overline{\hom(\C^{n_i},\C^{m_i})_{s_i}}=\{f:\C^{n_i}\fl{}\C^{m_i}\,|\,\rk(f)\leq s_i\}.$$\qed
\end{proof}

\section{$\mathfrak{g}$-Grassmannian.}

Let's study now the variety $Gr_{\mathfrak{g}}(V,S)$.

\begin{lema}
Let $V$ be an irreducible representation and let $s\leq n\in\N$ then
$$Gr_\mathfrak{g}(V^n,V^s)\cong Gr(n,s).$$
\end{lema}
\begin{proof}
A subrepresentation of $V^n$ isomorphic to $V^s$ determines and is determined by a subspace
of dimension $s$ inside the $n$ dimensional vector space of highest weight vectors of $V^n$.\qed
\end{proof}

\begin{cor}
Given $S\subseteq V$ with
$cl(V)=\sum_{i=1}^k n_icl(V_i)$ and $cl(S)=\sum_{i=1}^k s_icl(V_i)$.
$$Gr_{\mathfrak{g}}(V,S)=Gr_{\mathfrak{g}}(V_1^{n_1},V_1^{s_1})\times\ldots\times Gr_{\mathfrak{g}}(V_k^{n_k},V_k^{s_k}).$$
\end{cor}
\begin{proof}
Every subrepresentation $S\subseteq V$ is given by a morphism $S\mono V$, so it
decompose as a sum of $V_i^{s_i}\mono V_i^{n_i}$, in other words we have
subrepresentations $V_i^{s_i}\subseteq V_i^{n_i}$.\qed
\end{proof}

\begin{prop}
Given $S\subseteq V$ with
$cl(V)=\sum_{i=1}^k n_icl(V_i)$ and $cl(S)=\sum_{i=1}^k s_icl(V_i)$.
The variety $Gr_{\mathfrak{g}}(V,S)$ is irreducible, smooth and
$$\dim Gr_{\mathfrak{g}}(V,S)=\sum_{i=1}^k (n_i-s_i)s_i.$$
\end{prop}
\begin{proof}
We already know that the variety $Gr_{\mathfrak{g}}(V,S)$ is irreducible and smooth.
The dimension follows from \cite[p.138]{MR1182558},
$$\dim Gr(n,s)=(n-s)s.$$\qed
\end{proof}

\section{The variety $\mathcal{DGLA}(E)$.}

Recall the definition of a DGLA $(E,d,[-,-])$,
\begin{itemize}
\item $[E^i,E^j]\subset E^{i+j}$, $[a,b] = -(-1)^{\overline{a}\overline{b}}[b,a]$
\item $[a,[b,c]]=[[a,b],c]+(-1)^{\overline{a}\overline{b}}[b,[a,c]]$
\item $d(E^i)\subset E^{i+1}$, $d^2=0$, $d[a,b]=[da,b]+(-1)^{\overline{a}}[a,db]$
\end{itemize}
Note that $E^0$ is a Lie algebra, $E^i$ is a module ($i>0$), $d^0$ is a derivation (\cite[7.4.3]{MR1269324})
and that the bracket $E^1\times E^1\rightarrow E^2$ is symmetric.
In this section we are assuming $E=E^0\oplus E^1\oplus E^2$, so
a DGLA structure on $E$ is the data of a semisimple Lie algebra $E^0$,
two modules $E^1$ and $E^2$, two lineal maps $d^0$ and $d^1$ and
finally a symmetric bilinear $E^0$-morphism $f:S^2(E^1)\rightarrow E^2$,
$$E=E^0\oplus E^1\oplus E^2,\quad d^0:E^0\fl{}E^1,\,d^1:E^1\fl{}E^2,\,f:S^2(E^1)\fl{}E^2.$$
They satisfies the following conditions
$$d^0([a,b])=-b.d^0(a) +a.d^0(b),\quad d^1(a.x)=f(d^0a,x)+a.d^1(x),\quad d^1d^0=0.$$
Let's define the variety $\mathcal{DGLA}(E)$,
$$\mathcal{DGLA}(E)=\{(d^0,d^1,f)\,|\,\text{DGLA on } E\}.$$

\begin{notc}
For every $e\in \mathcal{DGLA}(E)$ we have its corresponding Maurer-Cartan variety $M(e)$
$$M(e):=\{x\in E^1\,|\,2d^1(x)+f(xx)=0\}.$$
We are using a lowercase $e$ to distinguish it from the graded vector space $E$.
There are a lot of structures of DGLA in the same space $E$.
\end{notc}

\begin{notc}
Let $\mathfrak{g}$ be a semisimple Lie algebra and let $V$ be a module.
The space of derivations $\der(\mathfrak{g},V)$ is representable by a module $\mathcal{I}$
in such a way that every derivation $d$ corresponds to a $\mathfrak{g}$-morphism $\phi_d:\mathcal{I}\rightarrow V$,
$$\der(\mathfrak{g},V)=\hom_\mathfrak{g}(\mathcal{I},V).$$
In this way we can associate to every derivation the $\mathfrak{g}$-rank of $\phi_d$,
$$\der(\mathfrak{g},V)\rightarrow R(\mathfrak{g}),\quad d\rightarrow \rk_{\mathfrak{g}}(\phi_d).$$
Recall that in the semisimple case all derivations are inner derivations then
$\der(\mathfrak{g},V)=V$. In other words, the notation $V_s$ will mean
$$V_s=\der(\mathfrak{g},V)_s=\{ d\in \der(\mathfrak{g},V)\,|\,\rk_{\mathfrak{g}}(\phi_d)=s\}.$$
\end{notc}

\begin{prop}
We have the following union of irreducibles, smooth subvarieties
$$\mathcal{DGLA}(E)=\bigcup_{0\leq v\leq cl(E^1)}\mathcal{DGLA}(E)_v,$$
where
$$\mathcal{DGLA}(E)_v=\{(d^0,d^1,f)\in\mathcal{DGLA}(E)\,|\,d^0\in E^1_v\},$$
$$\dim\mathcal{DGLA}(E)_v=\dim E^1_v+\dim\hom_{E^0}(E^1/v,E^2)+\dim\hom_{E^0}(S^2(E^1),E^2).$$
\end{prop}
\begin{proof}
The projection $\pi_1:\mathcal{DGLA}(E)_v\fl{}E^1_v$ has fibers
$$\pi_1^{-1}(d^0)=\{(d^1,f)\,|\,d^1(ax)=f(d^0a,x)+ad^1(x),\,d^1d^0=0\}.$$
Let $(d^1,f)\in\pi_1^{-1}(d^0)$ and recall that $f$ is an $E^0$-morphism
and that $d^0$ is an inner derivation (i.e. $d^0=-.y$ for some $y\in E^1$).
Let's analyze the conditions for $(d^1,f)$:
$$d^1(ax)=f(ay,x)+ad^1(x)=af(y,x)-f(y,ax)+ad^1(x)\iff $$
$$ad^1(x)+af(y,x)=d^1(ax)+f(y,ax)\iff
a(d^1+f(y,-))(x)=(d^1+f(y,-))(ax)\iff$$
$$d^1+f(y,-)\in\hom_{E^0}(E^1,E^2).$$
Projecting the variety $\pi_1^{-1}(d^0)$ over its second factor with $p_2$ and taking fiber, we get
$$p_2^{-1}(f)=\{ d^1\,|\,d^1+f(y,-)\in\hom_{E^0}(E^1,E^2),\,d^1d^0=0\}\cong$$
$$\{g\in\hom_{E^0}(E^1,E^2)\,|\,(g-f(y,-))d^0=0\}.$$
Let's rewrite the last condition
$$0=(g-f(y,-))d^0(a)=(g-f(y,-))(ay)=g(ay)-f(y,ay)=$$
$$ag(y)-\frac{a}{2}f(yy)=a(g(y)-\frac{1}{2}f(yy))\quad\forall a\in E^0.$$
Given that $E^0$ is semisimple, the last equation is equivalent to $2g(y)=f(yy)$.
Finally,
$$p_2^{-1}(f)\cong\{g\in\hom_{E^0}(E^1,E^2)\,|\,g(y)=f(yy)\}\cong \hom_{E^0}(E^1/v,E^2).$$
Note that
$$\im(p_2)=\hom_{E^0}(S^2(E^1),E^2).$$\qed
\end{proof}

We will see in the next two lemmas that every Maurer-Cartan variety of a DGLA $F$ with $F^0$ semisimple
appears in a DGLA $e\in\mathcal{DGLA}(E)_0$ where $E=F^0\oplus F^1\oplus \ker d^2$.
Note that $e$ has $d^0=0$.

\begin{lema}
Every degree zero map between two DGLA, $\phi:E\rightarrow F$, induces an equivariant map
$\phi_\star:M(E)\rightarrow M(F)$.
If $\phi_0$, $\phi_1$ are bijective and $\phi_2$ injective, $\phi_\star$ is an equivariant isomorphism.
\end{lema}
\begin{proof}
Let $x\in M(E)$ then $2dx+[x,x]=0$. Let's see what happens with $\phi(x)$,
$$2d(\phi(x))+[\phi(x),\phi(x)]=\phi(2dx+[x,x])=0,$$
then we have a well defined map $\phi_\star:M(E)\rightarrow M(F)$.

The simply connected Lie group associated to the Lie algebra $E^0$ acts on $E^1$
preserving the Maurer-Cartan variety. This action is called gauge-action,
$$a\cdot x=\sum_{k=0}^\infty \frac{a^k}{k!}(x-\frac{da}{k+1}).$$
Let's see that $\phi_\star$ preserve this action.
Let $a\in E^0$, $x\in M(E)$ and let $b=\phi_0(a)\in F^0$, $y=\phi_1(x)\in M(F)$.
Given that $\phi$ is a DGLA map, we have $\phi_1(d^0a)=d^0\phi_0(a)=d^0b$ then
$$\phi_\star(a\cdot x)=\phi_\star(\sum_{k=0}^\infty \frac{a^k}{k!}(x-\frac{da}{k+1}))=
\sum_{k=0}^\infty \frac{b^k}{k!}(y-\frac{db}{k+1})=b\cdot y.$$
Finally, suppose that $\phi_1$ is surjective and $\phi_2$ injective,
$$z\in M(F)\then 0=2dz+[z,z]=2d\phi(x)+[\phi(x),\phi(x)]=\phi(2dx+[x,x])\then x\in M(E)$$
then $\phi_\star$ is surjective. If $\phi_1$ is injective, clearly so is $\phi_\star$.\qed
\end{proof}

\begin{lema}
Let $E=E^0\oplus E^1\oplus E^2$ with $E^0$ semisimple and $E^1$, $E^2$ two modules.
For every $e\in\mathcal{DGLA}(E)$ there exist $\widetilde{e}\in\mathcal{DGLA}(E)_0$ such that
$$M(e)\cong M(\widetilde{e}).$$
\end{lema}
\begin{proof}
Given $x\in M(e)$ we can construct another DGLA structure on $E$ different from $e=(d^0,d^1,f)$,
$$\widetilde{e}:=(d^0-\mu_x,d^1+f(x,-),f)$$
where $\mu_x$ is the inner derivation $\mu_x(a)=a.x=-[x,a]$.
The Maurer-Cartan variety $M(\widetilde{e})$ is isomorphic to $M(e)$ in an equivariant way:
$$M(\widetilde{e})\fl{}M(e),\quad z\rightarrow z+x.$$
$$(a\cdot_2 z)+x=\sum_{k=0}^\infty \frac{a^k}{k!}(z-\frac{da-ax}{k+1})+x=
\sum_{k=0}^\infty \frac{a^k}{k!}(z-\frac{da}{k+1})
+\sum_{k=0}^\infty \frac{a^k}{k!}\frac{ax}{k+1}+x=$$
$$\sum_{k=0}^\infty \frac{a^k}{k!}(z-\frac{da}{k+1})
+\sum_{k=0}^\infty \frac{a^k}{k!}x=\sum_{k=0}^\infty \frac{a^k}{k!}(z+x-\frac{da}{k+1})=a\cdot_1(z+x).$$
Given that $d^0$ is an inner derivation there exist $y\in E^1$ such that $d^0=\mu_y$,
then if $y\in M(e)$, we can make $\widetilde{e}\in\mathcal{DGLA}(E)_0$.
Let $e=(\mu_y,d^1,f)\in\mathcal{DGLA}(E)$, let's see that $y\in M(e)$,
$$0=d^1\mu_y(a)=d^1(ay)=f(ay,y)+ad^1(y)=$$
$$\frac{1}{2}af(y,y)+ad^1(y)=a(d^1(y)+\frac{1}{2}f(y,y))\quad\forall a\in E^0.$$
Given that $E^0$ is semisimple, $y\in M(e)$.\qed
\end{proof}

\begin{obs}
From the previous two lemmas we know that every Maurer-Cartan variety associated to a DGLA $E$ with
$E^0$ semisimple could be obtained from a structure of the form $e=(0,d^1,f)$.
We want to mention here that $e$ induce an $E^0$-morphism
$$e:E^1\oplus S^2(E^1)\rightarrow E^2,\quad e(x,q)=2d^1(x)+f(q).$$
We already noted before that $f$ is an $E^0$-morphism and in this case, when $d^0=0$, we get that
$d^1$ is also an $E^0$-morphism (see the DGLA conditions). In fact, we have
$$\mathcal{DGLA}(E)_0\cong \hom_{E^0}(E^1\oplus S^2(E^1),E^2).$$
For every submodule isomorphic to $E^2$ inside $E^1\oplus S^2(E^1)$ we will have a
variety of Maurer-Cartan.
\end{obs}

\section*{Acknowledgments.}
I would like to thanks my advisor, Fernando Cukierman for the very useful discussions, ideas and suggestions.
All the ideas presented here were given by him. This work was supported by CONICET, Argentina.
\bibliographystyle{elsarticle-num}
\bibliography{../../citas}

\end{document}